\documentclass[10pt]{amsart}

\usepackage{amssymb}
\usepackage{epsfig}   

\usepackage[foot]{amsaddr}
\usepackage{caption}
\usepackage{xcolor}

\usepackage{setspace}
\usepackage{parskip}
\newtheorem{thm}{Theorem}[section]
\newtheorem{prop}[thm]{Proposition}
\newtheorem{lem}[thm]{Lemma}
\newtheorem{cor}[thm]{Corollary}
\newtheorem{conj}[thm]{Conjecture}

\newtheorem{claim}[thm]{Claim}

\theoremstyle{definition}
\newtheorem{definition}[thm]{Definition}

\theoremstyle{remark}

\numberwithin{equation}{section}

\newcommand{\R}{\mathbb{R}}  
\newcommand{\Q}{\mathbb{Q}}

\newcommand{\N}{\mathbb{N}}

\newcommand{\norm}[1]{\left\lvert#1\right\rvert}

\newcommand{\Homeo}{\textrm{Homeo}}  
 
\newcommand{\Exp}{\textrm{Exp}} 
\newcommand{\aut}{\textrm{Aut}}

\newcommand{\rank}{\textrm{rk}}
\newcommand{\proj}{\textrm{proj}}
\newcommand{\ch}{\textrm{Ch}}

\begin{document}

\title[]{A Ramsey-type phenomenon in two and three dimensional simplices}

\author{Sumun Iyer}
\address{Cornell University}
\email{ssi22@cornell.edu}

\begin{abstract}
We develop a Ramsey-like theorem for subsets of the two and three-dimensional simplex. A generalization of the combinatorial theorem presented here to all dimensions would produce a new proof that $\Homeo_+[0,1]$ is extremely amenable (a theorem due to Pestov) using general results of Uspenskij on extreme amenability in homeomorphism groups. 
\end{abstract}

\maketitle

\bigskip



\section{Introduction}
This note develops a Ramsey-type phenomenon about subsets of the two- dimensional simplex:
\begin{thm}\label{thmintro}
Let $F$ be a subset of $\Delta^2=\{(x,y) \ : \ 0\leq x\leq y\leq 1\}$. Then one of the following is true:
\begin{enumerate}
    \item $F^c$ has infinite rank
    \item any finite pattern is essential for $F$.
\end{enumerate}
\end{thm}
\noindent Precise definitions of \emph{rank} and \emph{essential patterns} are in Section \ref{sec2ramsey2splx}. For now we just mention that, loosely, the theorem says that any subset of the two-simplex is either geometrically ``thin'' or it is combinatorially very rich. Theorem \ref{thmintro} is proven in Section \ref{sec2ramsey2splx} (as Theorem \ref{thm1}); the proof is combinatorial but the moves are inspired by some geometric intuitions. The main interest of Theorem \ref{thmintro} is that it seems to indicate a new sort of Ramsey phenomenon. The theorem statement has a Ramsey-like structure and there are Ramsey theoretic gestures throughout the proof, but there does not appear to be any direct relation to existing Ramsey theorems. The motivation for formulating and proving Theorem \ref{thmintro} comes from topological dynamics and we say a bit about that now.

A topological group is \emph{extremely amenable} if every continuous action of it on a compact Hausdorff spaces has a fixed point. For the class of Polish groups which arise as automorphism groups of Fraissé limits, extreme amenability is related to the structural Ramsey property in classes of finite structures (this is known as the \emph{Kechris-Pestov-Todorcevic correspondence}, see \cite{kpt} for more). We are interested here in extreme amenability outside the setting of Fraissé limit automorphism groups.

Let $\Homeo_+[0,1]$ be the group of orientation-preserving homeomorphisms of the unit interval. First, note that $\Homeo_+[0,1]$ is known to be extremely amenable; this is a theorem of Pestov \cite{pestov}. The standard proof goes like this: one notes that there is a map from the group $\aut (\Q,\leq)$ of order-preserving bijections of the rationals into $\Homeo_+[0,1]$ with dense image. Then an application of the classical Ramsey theorem and the Kechris-Pestov-Todorcevic correspondence implies that $\aut (\Q,\leq)$ is extremely amenable. It is a general and easy to see fact that if a group $G$ has an extremely amenable, dense subgroup, then $G$ itself must be extremely amenable.

Let $X$ be a compact metric space and $\Homeo (X)$ the group of homeomorphisms of $X$ with the uniform convergence topology. A theorem of Uspenskij in \cite{uspenskij} shows that determining extreme amenability of $\Homeo(X)$ is equivalent to proving that for a certain countable family of $\Homeo(X)$-flows, the only minimal subflows are singletons. The construction in \cite{uspenskij} suggests a possible new approach to proving that $\Homeo_+[0,1]$ is extremely amenable, as Uspenskij notes. For $n \in \N$, let
\[\Delta^n=\{(x_1,x_2,\ldots, x_n) \ : \ 0 \leq x_1\leq x_2\leq \cdots x_n \leq 1\}\]
be the standard realization of the $n$-dimensional simplex. The group $\Homeo_+[0,1]$ acts on $\Delta^n$ diagonally: 
\[g \cdot (x_1,x_2,\ldots, x_n)=(g(x_1),g(x_2),\ldots, g(x_n))\]
A consequence of Uspenskij's general theorem is that extreme amenability of \\ $\Homeo_+[0,1]$ is equivalent to the statement below:
\begin{lem}(Lemma 1.2 of \cite{uspenskij})\label{thmusplem}
For each $n \in \N$, every minimal, closed, $\Homeo_+[0,1]$-invariant subset of $\Exp (\Delta^n)$ is a singleton.
\end{lem}
\noindent Throughout, for $K$ compact, $\Exp (K)$ is the compact space of all compact subsets of $K$ with the Vietoris topology (see \cite{kechris}, Section 4.F for definitions).

So a proof of Lemma \ref{thmusplem} (that of course, does not simply apply Pestov's theorem) would provide a new proof that $\Homeo_+[0,1]$ is extremely amenable. Uspenskij mentions that he's ``not aware of a short, independent proof of the lemma," and this note began as an attempt to bridge this gap. We succeed in providing an independent proof for dimension $n=2$ using Theorem \ref{thmintro}. The definitions needed to state Theorem \ref{thmintro} and the proof of Theorem \ref{thmintro} occupy Section \ref{sec2ramsey2splx}. We note that Section \ref{sec2ramsey2splx} makes no mention of topological dynamics or of the group $\Homeo_+[0,1]$, it can be read completely on its own. In Section \ref{sec3dynamics}, we return to topological dynamics and show how Theorem \ref{thmintro} implies Uspenskij's lemma (Lemma \ref{thmusplem} above) for $n=2$.

In Section \ref{sec43splx}, we formulate a Ramsey statement which is a natural generalization of Theorem \ref{thmintro} for all dimensions (Conjecture \ref{conj1}); an independent proof of this statement would answer Uspenskij's question. We then prove a partial result (Theorem \ref{thm3}) towards Conjecture \ref{conj1} for the simplex of dimension 3. There are some interesting technical difficulties that arise in dimension 3 as opposed to dimension 2 and which may indicate some sort of new approach is needed to prove Conjecture \ref{conj1} in all dimensions.

\subsection{Acknowledgements:} I would like to thank Sławek Solecki for many helpful conversations about this project and also for encouraging me to continue to work on it.


\section{A Ramsey theorem for the 2-simplex}\label{sec2ramsey2splx}
By $\Delta^2$, we denote the usual geometric realization of the two-simplex given by:
\[\Delta^2=\{(x,y) \ : \ 0\leq x\leq y\leq 1\}\subseteq \R^2.\]

This first definition assigns a rank to subsets of $\Delta^2$ that captures how ``thick'' they are.

\begin{definition}\label{def1}
For $S \subseteq \Delta^2$, say that \emph{rank of $S$ is at least $n$} if there exists 
\[0\leq x_0<x_1\leq y_1<x_2\leq y_2<\cdots <x_n\leq y_n<y_n' \leq 1\]
such that 
\[S \supseteq (x_0,x_1) \times (y_1,y_n') \cup (x_0,x_2)\times (y_2,y_n') \cup \cdots \cup (x_0,x_n)\times (y_n,y_n') \]
\end{definition}

We will denote by $\rank(S)$ the greatest integer $n$ so that the rank of $S$ is at least $n$. If $S$ has rank at least $n$ for all $n \in \N$, then we write $\rank(S)=\infty$. Clearly, this rank is monotone in the sense that 
\[A \subset B \implies \rank(A) \leq \rank(B)\]
Notice also that $\rank(\Delta^2)=\infty$ and in fact if we let 
\[D_{a,b}=\{(x,y) \in \Delta^2 \ : \ a<x<y<b\}\]
for any $a<b$, then $\rank(D_{a,b})=\infty$. These open triangles of the form $D_{a,b}$ will be important later on, and we refer to them as \emph{essential triangles}. 

The condition of being ``combinatorially rich'' is captured by the notion of \emph{patterns}. For a natural number $k$, we use $[k]$ to denote the set $\{0,1,\ldots,k-1\}$. For a set $A$ and a natural number $i$, $(A)^i$ is the collection of all subsets of $A$ of cardinality $i$. For $n$ an even natural number, a \emph{pattern of size $n$} is a subset $P$ of $([n])^2$ so that $\bigcup_{a\in P}a=[n]$ and $a\cap b=\emptyset$ for each $a,b \in P$ with $a \neq b$. 

Let $O:(\N)^2 \to \N \times \N$ be the map:
\[O(\{a,b\})=\begin{cases}
(a,b) & \textrm{ if }a < b\\
(b,a)& \textrm{ if }a>b\\
\end{cases}\]
For $i=1,2$, let $\proj_i :\N \times \N\to \N$ be the projection onto the $i$th coordinate. We also use $\proj_1:\Delta^2 \to [0,1]$ to be the projection onto the $x$-coordinate and $\proj_2:\Delta^2 \to [0,1]$ to be the projection onto the $y$-coordinate (the meaning should always be clear from context).

If $P$ is a pattern of size $n$ and $F\subseteq \Delta^2$, then a \emph{copy of $P$ in $F$} is a function
\[\sigma: O(P) \to F\]
such that for any $a,b \in O(P)$ and $p,q \in \{\proj_1,\proj_2\}$
\[p(a)<q(b) \implies p(\sigma(a))<q(\sigma(b))\]
Notice that a copy $\sigma$ of $P$ induces an order preserving function which we denote by $f_\sigma:[n] \to [0,1]$ defined by:
\[f_\sigma(j)=\begin{cases}
\proj_1(\sigma(a)) & \textrm{ if }j=\proj_1(a)\\
\proj_2(\sigma(a)) &\textrm{ if }j=\proj_2(a)\\
\end{cases}
\]
We say that a pattern $P$ is \emph{essential} for $F$ if there is a copy of $P$ in $D_{a,b} \cap F$ for every $a<b$. Here is the statement of the main theorem:

\begin{thm}\label{thm1}
Let $F$ be a subset of $\Delta^2$. Then one of the following is true:
\begin{enumerate}
    \item $\rank(F^c)=\infty$
    \item any finite pattern is essential for $F$.
\end{enumerate}
\end{thm}

First let $\psi_k^n:[n] \to [n+1]$ be given by:
\[\psi_k^n(x)=\begin{cases}
x & \textrm{ if }x<k\\
x+1 & \textrm{ if }x\geq k\\
\end{cases}
\]
We also have an inverse operator $(\psi_k^n)^{-1}:[n+1]\setminus \{k\} \to [n]$:
\[(\psi_k^n)^{-1}(x)=\begin{cases}
x & \textrm{ if }x<k\\
x-1 &\textrm{ if }x\geq k+1\\
\end{cases}\]

Both maps are injective and order-preserving on their respective domains. Abusing notation slightly, we use $\psi_k^n$ and $(\psi_k^n)^{-1}$ also to represent the obvious maps induced on $(\N)^2$ and on $\N \times \N$ coordinate-wise.

\begin{lem}\label{lem1}
Assume that $F$ is a subset of $\Delta^2$ and there exists $M \in \N$ so that for any set $S\subseteq \Delta^2$ with $\rank(S) \geq M$, $S \cap F\neq \emptyset$. If $P$ is a  pattern of size $n$ that is essential for $F$, then for any $k \leq n$, $\psi_k^n(P) \cup \{(k,n+1)\}$ is a pattern of size $n+2$ that is essential for $F$. 
\end{lem}

\begin{proof}
Let $P' = \psi_k^n(P) \cup \{k,n+1\}$; we want to prove that $P'$ is essential for $F$.

We first take care of the case that $k=n$. Fix $a<b$ and since $P$ is essential, let $\sigma:O(P) \to F \cap D_{a,b}$ be a copy of $P$. Let $f_\sigma: [n] \to [0,1]$ be the associated order-preserving function. Note that $f_\sigma (n-1) <b$ since $\sigma(O(P)) \subseteq D_{a,b}$. The essential triangle $D_{f_\sigma(n-1),b}$ has rank $\infty$ and so let $(x,y) \in F \cap D_{f_\sigma(n-1),b}$. Then, one can see that $\tau:O(P') \to F \cap D_{a,b}$ defined by
\[\tau (a)=
\begin{cases}
 \sigma(a) & \textrm{ if }a \in O(P)\\
 (x,y) & \textrm{ if }a=(n,n+1)\\
\end{cases}\]
is a copy of $P'$ in $F \cap D_{a,b}$ since $y>x>f_\sigma(n-1)$. So $P'$ is essential for $F$.

Now assume that $k<n$. Fix $a<b$ and suppose for contradiction that $F \cap D_{a,b}$ does not contain a copy of $P'$. Set $a_0=a$ and $b_0=b$. We will inductively define $(a_i,b_i,c_i)$ for $i=1,\ldots,M+2$ so that for each $i \geq 1$:
\begin{enumerate}
    \item $a_{i-1}<a_i<b_i\leq c_i<b_{i-1}$
    \item $F\cap (a_i,b_i) \times (c_i,b) =\emptyset$
\end{enumerate}

Let $\sigma_1:O(P) \to F \cap D_{a,b}$ be a copy of $P$ in $F \cap D_{a,b}$. Let $f_{\sigma_1}:[n] \to [0,1]$ be the order-preserving function associated to $\sigma_1$. Set $a_1=f_{\sigma_1}(k-1)$, $b_1=f_{\sigma_1}(k)$, and $c_1=f_{\sigma_1}(n-1)$. Condition (1) is satisfied because $f_{\sigma_1}$ is order-preserving and since $\sigma(O(P))\subset D_{a,b}$, it must be that $f_{\sigma_1}([n]) \subset (a,b)$. To see that condition (2) is satisfied, suppose for contradiction that $(x,y) \in F\cap (a_1,b_1) \times (c_1,b)$. Define $\sigma_1':O(P') \to F\cap D_{a,b}$ by
\[\sigma_1'(z)=\begin{cases}
\sigma_1(w), & \textrm{ if }z=\psi_k^n(w)\\
(x,y), & \textrm{ if }z=(k,n+1)\\
\end{cases}
\]
Now one can check that $\sigma_1'$ is a copy of $P'$ in $F \cap D_{a,b}$, exactly because $(x,y) \in (a_1,b_1)\times (c_1,b)=(f_{\sigma_1}(k-1),f_{\sigma_1}(k)) \times (f_{\sigma_1}(n-1),b)$. This cannot happen and so it must be that $F\cap (a_1,b_1)\times (c_1,b)=\emptyset$, i.e., that condition (2) holds. 
Suppose now that we have defined $(a_n,b_n,c_n)$ satisfying (1) and (2). As $P$ is essential for $F$, let $\sigma_{n+1}:O(P)\to F \cap D_{a_n,b_n}$ be a copy of $P$ and let $f_{\sigma_{n+1}}:[n] \to [0,1]$ be the associated order-preserving function. Set $a_{n+1}=f_{\sigma_{n+1}}(k-1)$, $b_{n+1}=f_{\sigma_{n+1}}(k)$, and $c_{n+1}=f_{\sigma_{n+1}}(n-1)$. Condition (1) for $i=n+1$ is satisfied because $\sigma_{n+1}(O(P)) \subset D_{a_n,b_n}$ and condition (2) is satisfied because if $ (x,y) \in F \cap (a_{n+1},b_{n+1}) \times (c_{n+1},b)$ then just as in the base case one can check that 
\[\sigma_{n+1}'(z)=\begin{cases}
\sigma_{n+1}(w), & \textrm{ if }z=\psi_k^n(w)\\
(x,y), & \textrm{ if }z=(k,n+1)\\
\end{cases}\]
is a copy of $P'$ in $F \cap D_{a,b}$. 

Once $(a_M+2,b_M+2,c_M+2)$ has been defined, condition (2) implies that 
\begin{multline*}
    R:= (a_{M+2},b_{M+2})\times (c_{M+2},b) \cup (a_{M+2},b_{M+1}) \times (c_{M+1},b) \cup \\
    \cdots \cup (a_{M+2},b_1) \times(c_1,b) \subset F^c 
\end{multline*}

and condition (1) implies that
\[a_{M+2}<b_{M+2}\leq c_{M+2}<b_{M+1}\leq c_{M+1}<\cdots <b_1\leq c_1 <b\]
and so $\rank(R) \geq M$ yet $R \cap F=\emptyset$, which contradicts our assumption on $F$. So $P'$ is essential for $F$.
\end{proof}

Now we can prove Theorem \ref{thm1}:

\begin{proof}[Proof of Theorem \ref{thm1}]
Suppose that $F$ is a subset of $\Delta^2$ and that $\rank(F^c)=M<\infty$. Then, every subset of $S \subseteq \Delta^2$ with $\rank(S) \geq M+1$ is such that $S \cap F\neq \emptyset$. So we satisfy the assumptions of Lemma \ref{lem1}. 

Now we want to argue by induction on the size of patterns that any finite pattern is essential for $F$. First the base case. It is clear that the pattern $\{0,1\}$ is essential for $F$ since $F$ must intersect every essential triangle (recall that essential triangles have infinite rank) and $\{0,1\}$ is the only pattern of size $2$. Suppose now that every pattern of size $n$ is essential for $F$. Let $P$ be a pattern of size $n+2$. Let $a \in P$ be the member of $P$ with $n+1 \in a$ and suppose that $a=\{k,n+1\}$ where of course $k\leq n$. Note that $(\psi_k^n)^{-1}\left(P \setminus\{a\}\right)$ is a pattern of size $n$ and so is essential for $F$. Then, Lemma \ref{lem1} implies that
\[P=\psi_k^n\left((\psi_k^n)^{-1}\left(P \setminus\{a\}\right)\right)\cup \{(k,n+1\}\]
is essential for $F$.
\end{proof}



\section{Dynamical consequences}\label{sec3dynamics}

We show in this section how Theorem \ref{thm1} implies Theorem \ref{thmusplem} in dimension 2. Recall that we consider the diagonal action of $\Homeo_+[0,1]$ on $\Delta^2$:
\[g\cdot (x,y)=(g(x),g(y)) \textrm{ for }(x,y) \in \Delta^2 \textrm{ and }g \in \Homeo_+[0,1]\]
For $A \subset \R^2$ and $\epsilon >0$, define
\[(A)_\epsilon=\{(x,y) \in \R^2 \ : \ d((x,y),a)<\epsilon \textrm{ for some }a \in A\}\]
where $d$ is the Euclidean metric on $\R^2$. By $\partial \Delta^2$, we mean the boundary of $\Delta^2$--the union of the one-dimensional faces of $\Delta^2$.

First, here are two definitions capturing the properties of subsets of $\Delta^2$ that we will be interested in. We say $F \subseteq \Delta^2$ is \emph{pseudo-dense} if for any $\epsilon >0$, there exists $g \in \Homeo_+[0,1]$ so that $g(F)$ is $\epsilon$-dense in $\Delta^2$. We say $F \subseteq \Delta^2$ is \emph{thin} if for any $\epsilon >0$ there exists some $g \in \Homeo_+[0,1]$ so that $g(F) \subseteq \left(\partial \Delta^2\right)_\epsilon$. Lemma \ref{lem2} below proves equivalent conditions for each of these two properties that connects them to Theorem \ref{thm1}.

Lemma \ref{lem3} is easy to check from the definition of rank and the fact that elements of $\Homeo+[0,1]$ preserve the relative orders of sets of points in $[0,1]$.

\begin{lem}\label{lem3}
For any set $S \subset \Delta^2$ and any $f \in \Homeo_+[0,1]$, $\rank(f(S))=\rank(S)$.
\end{lem}

\begin{lem}\label{lem2}
Let $F$ be a subset of $\Delta^2$. Then:
\begin{enumerate}
    \item $F$ is pseudo-dense if and only if $F$ contains a copy of every finite pattern.
    \item $F$ is thin if and only if $\rank(F^c)=\infty$.
\end{enumerate}
\end{lem}

It is convenient here to define a particular type of pattern. A pattern $P$ is a \emph{grid of width $n$} if 
\[O(P)=\{(x_i^j,y_i^j) \ : \ 0\leq i\leq j\leq n-1\}\]
satisfying the condition that
\begin{align}\label{eqn3}
    \{x_1^1, \ldots, x_1^n, y_1^1\} &<\{x_2^2, \ldots, x_2^n,y_1^2,y_2^2\} < \cdots \\
    &< \{x_j^j,\ldots, x_j^n, y_1^j, y_2^j,\ldots , y_j^j\} < \cdots < \{x_n^n,y_n^n\} \notag
\end{align}
where we write $A<B$ for finite subsets $A,B$ of $\N$ if $\max(A)<\min(B)$.
Another formulation: a pattern $P$ is a grid of width $n$ if there exists reals 
\[0<x_1<y_1<x_2<y_2<\cdots <x_{n+1}<y_{n+1}<1\]
so that $\norm{O(P) \cap (x_i,x_{i+1}) \times (y_j,y_{j+1})}=1$ for each $1 \leq i\leq j\leq n$. It is not hard to see with a moment of thought that for any finite pattern $P$, there exists some $N$ so that any copy of a grid of width $N$ contains a copy of $P$. 

\begin{proof}[Proof of Lemma \ref{lem2}]
\textbf{(1):} First suppose that $F$ is pseudo-dense. Fix $n \in \N$. Let $g \in \Homeo_+[0,1]$ be so that $g(F)$ is $\frac{1}{4n}$-dense in $\Delta^2$. So for $0 \leq i\leq j \leq n-1$, $g(F)$ intersects the set $\left(\frac{i}{n},\frac{i+1}{n}\right) \times \left(\frac{j}{n},\frac{j+1}{n}\right)$ and we can choose a point $(a_i^j,b_i^j)$ in the intersection. Let $P=\{(x_i^j,y_i^j) \ : \ 0\leq i\leq j\leq n-1\}$ be a grid of width $n$ with the $x_i^j,y_i^j$ satisfying Condition \ref{eqn3}. It is then not hard to check that $\sigma:O(P) \to F$ given  by
\[\sigma \left((x_i^j,y_i^j)\right)=g^{-1} \cdot (a_i^j,b_i^j)\]
is a copy of $P$ in $F$, since $g^{-1}$ is order-preserving. Since $F$ contains a copy of grids of arbitrary width, $F$ contains a copy of every finite pattern. 

Conversely, suppose $F$ contains a copy of every pattern. Then, given $\epsilon >0$, let $n$ be so that $\frac{4}{2n+3}\sqrt{2}<\epsilon$. Let $\sigma:O(P) \to F$ be a copy of $P$ in $F$, where $P$ is a grid of width $n$. Let $x_1,y_1,\ldots, x_{n+1},y_{n+1}$ be as in the second formulation of what is means to be a grid. Choose $g \in \Homeo_+[0,1]$ so that $g(x_i)=\frac{2i-1}{2n+3}$ and $g(y_i)=\frac{2i}{2n+3}$ for $1\leq i\leq n+1$. Then, $g(F)$ has the property that for any odd $i$ and even $j$ with $2n+2\geq j>i\geq 1$, the square $\left(\frac{i}{2n+3},\frac{i+2}{2n+3}\right)\times \left(\frac{j}{2n+3},\frac{j+2}{2n+3}\right)$ intersects $g(F)$ this implies that $g(F)$ is $\epsilon$-dense in $\Delta^2$.

\textbf{(2):} Suppose that $F$ is thin and fix $\epsilon >0$. Let $g$ be in $\Homeo_+[0,1]$ so that $g(F) \subset \left(\partial 
\Delta^2\right)_\epsilon$. Let $k \in \N$ satisfy
\begin{equation}\label{eqn0}
    (2k+3)\epsilon <1-\epsilon
\end{equation}
Then, observe that
\[
\begin{split}
(\epsilon,2\epsilon) & \times (3\epsilon, (2k+3)-\epsilon) \cup (\epsilon, 4\epsilon) \times (5\epsilon,(2k+3)-\epsilon) \cup \cdots \\
\cup & (\epsilon,2k\epsilon) \times ((2k+1)\epsilon,(2k+3)\epsilon) \subseteq \left(g(F)\right)^c
\end{split}
\]
and so $\rank(g(F)^c)\geq k$. By Lemma \ref{lem3}, $\rank(F^c)\geq k$. By making $\epsilon$ small, we can find choices of $k$ satisfying \ref{eqn0} that as as large as we wish; so $\rank(F^c)=\infty$.

Before we prove the converse, notice the following: if $(x,y) \in \Delta^2$ is such that one of the three conditions below holds:
\begin{enumerate}
    \item $x<\epsilon$
    \item $y>1-\epsilon$
    \item $y-x<\epsilon$
\end{enumerate}
then, $(x,y) \in \left(\partial \Delta^2\right)_\epsilon$.

Now, suppose $\rank(F^c) \geq n$. Let $\beta=\frac{1}{2n+4}$.  Let: 
\[x_0 <x_1\leq y_1<\cdots <x_n\leq y_n <y_n'\]
be as in Definition \ref{def1}, witnessing that $\rank(F^c) \geq n$. Let $f \in \Homeo_+[0,1]$ so that $f(x_0)=\beta$; $f(x_i)=2i\beta$ for $1\leq i\leq n$; $f(y_i)=(2i+1)\beta$ for $1\leq i\leq n$; and $f(y_n')=\frac{2n+3}{2n+4}=1-\beta$. Then suppose that $(x,y) \notin f(F^c)$. This implies
\[(x,y) \notin (\beta,2\beta) \times (3\beta,1-\beta) \cup (\beta,4\beta)\times (5\beta,1-\beta) \cup \cdots \cup (\beta,2n\beta)\times ((2n+1)\beta,1-\beta)\]
Now, there are three options for $x$: (1) $x<\beta$, (2) $x>2n\beta$, or (3) $2k\beta \leq x\leq (2k+2)\beta $ for some $0\leq k\leq n$. Of course, (2) implies that $y >2n\beta =1-4\beta$. Case (3) implies that either $y<(2k+3)\beta$ in which case $y-x<(2k+3)\beta-2k\beta=3\beta$ or $y>1-\beta$. So by our observation above, $(x,y) \in \left(\partial \Delta^2\right)_{4\beta}$ and thus that $f(F^c)^c =f(F) \subseteq \left(\partial \Delta^2\right)_{4\beta}$. It follows that $\rank(F^c) =\infty$ implies that $F$ is thin.
\end{proof}

The following fact about the topology of $\Exp(Y)$ will be useful: for $F \subseteq \Exp(Y)$ and $A \in \Exp(Y)$, $A \in \overline{F}$ if for every $\epsilon >0$, there exists $X \in F$ so that $A \subseteq (X)_\epsilon$ and $X \subseteq (A)_\epsilon$. 

\begin{prop}\label{prop1}
Every minimal subflow of $\Exp(\Delta^2)$ is a singleton. Moreover, it is a union of faces of $\Delta^2$.
\end{prop}

First a lemma which illustrates the method of proving Proposition \ref{prop1} in the easy one-dimensional case and will be used in the proof:

\begin{lem}\label{lem4}
Every minimal subflow of $\Exp([0,1])$ is a singleton. Actually, it is a union of faces of the one-simplex $[0,1]$, that is, one of:
\[\{[0,1]\},\{0\},\{1\}, \{0,1\}\]
\end{lem}

\begin{proof}
Let $X$ be a minimal subflow of $\Exp([0,1])$ and let $A \in X$. There are a few cases to consider. Suppose first that there is an interval $(a,b)$ contained in $A^c$ and that $A$ contains points less than $a$ and greater than $b$. Then, for any $\epsilon >0$, there is some $g \in \Homeo_+[0,1]$ so that $g(a)=\frac{\epsilon}{2}$ and $g(b) =1-\frac{\epsilon}{2}$. Notice that for such a $g$, $g(A) \subseteq \left(\{0,1\}\right)_\epsilon$ and $\{0,1\} \subset (g(A))_\epsilon$. Since $X$ is $\Homeo_+[0,1]$-invariant and closed in $\Exp([0,1])$, it follows that $\{0,1\} \in X$. But $\{0,1\}$ is clearly a fixed point of the action of $\Homeo_+[0,1]$ on $\Exp([0,1])$ and $X$ is minimal; so $X =\{\{0,1\}\}$. Similar arguments in the cases that $A \subset [b,1]$ (resp. $A\subset [0,a]$) give that $X=\{\{1\}\}$ (resp. $X=\{\{0\}\}$). 

If there is no such interval $(a,b) \subset A^c$; then $A$ is dense in $[0,1]$ and it follows that $A=[0,1]$ and so $X=\{\{[0,1]\}\}$ since $[0,1]$ is a fixed point of the action.
\end{proof}

Now we can prove the proposition:

\begin{proof}[Proof of Proposition \ref{prop1}]
Let $X$ be a minimal subflow of $\Exp(\Delta^2)$. Let $A \in X$. Lemma \ref{lem2} and Theorem \ref{thm1} imply that at least one of the following two things happens:

\begin{enumerate}
    \item $A$ is pseudo-dense.
    \item $A$ is thin.
\end{enumerate}
Since $X$ is $\Homeo_+[0,1]$-invariant, $g(A) \in X$ for each $g \in \Homeo_+[0,1]$. In the case that (1) occurs, the fact that $X$ is a closed subset of $\Exp(\Delta^2)$ implies that $\Delta^2 \in X$. It is clear that $\Delta^2$ is a fixed point of the action of $\Homeo_+[0,1]$ and since $X$ is minimal it must be that $X=\{\Delta^2\}$.

In the case that (2) occurs, compactness of $X$ implies that there exists some $B \in X$ so that $B \subseteq \partial \Delta^2$. The boundary of $\Delta^2$ is composed of three line segments. Consider for example 
\[B_0:=B \cap \{(x,y) \ : \ x=0 \textrm{ and }0\leq y\leq 1 \}\]

By Lemma \ref{lem4}, there exists some $B' \in X$ so that $B' \subseteq \partial \Delta^2$ (notice that each face of $\Delta^2$ is preserved set-wise by the action) and so that 
\[B' \cap \{(x,y) \ : \ x=0 \textrm{ and }0\leq y\leq 1 \}\]
is some union of faces of $\{(x,y) \ : \ x=0 \textrm{ and }0\leq y\leq 1 \}$.

Repeating this argument two more times on the other two one-dimensional faces of $\Delta^2$, applying in turn the fact that the action of $\Homeo_+[0,1]$ preserves each face set-wise, we produce $C \in X$ so that $C \subseteq \partial \Delta^2$ is a union of faces of $\Delta^2$. Since $C$ is a fixed point of the action, $X=\{C\}$.
\end{proof}

\section{A 3-dimensional theorem}\label{sec43splx}

We develop a natural analog of the notion of rank in two dimensions to higher dimensions, conjecture the natural generalization of Theorem \ref{thm1} to higher dimensions, and then prove a theorem for the 3-simplex that is partial progress towards resolving the conjecture. We take the usual representation of the $n$-simplex:
\[\Delta^n=\{(x_1,x_2,\ldots, x_n) \ : \ 0\leq x_1\leq x_2\leq \cdots \leq x_n \leq 1\} \subset \R^n\]

\begin{definition}\label{def2}
A set $A \subset \Delta^n$ has \emph{rank $\geq m$} if there exists 
\[
\begin{split}
0 & \leq x_1^0 <x_1^1 \leq x_2^1\leq x_3^1 \leq \cdots \leq x_n^1 < x_1^2 \leq x_2^2 \leq \cdots \\
& \leq x_n^2< \cdots < x_1^m \leq x_2^m \leq \cdots \leq x_n^m < x_n^{m+1} \leq 1
\end{split}
\]
such that:
\[\textrm{int} \left( \bigcup_{0\leq i_1<i_2<\cdots <i_n \leq m} [x_1^{i_1},x_1^{i_1+1}] \times [x_2^{i_2},x_2^{i_2+1}] \times \cdots \times [x_n^{i_n},x_n^{i_n+1}] \right) \subseteq A. \]
\end{definition}

Just as in the two dimensional case, we define patterns. For $N \in \N$ where $n \vert N$, an \emph{$n$-pattern of size $N$} is a subset $P$ of $([N])^n$ such that $\bigcup_{a\in P}a=[N]$ and $a \cap b =\emptyset$ for each $a,b \in P$ with $a \neq b$. Let
\[O_n:(\N)^n \to  \underbrace{\N \times \cdots \times \N}_{n\textrm{-times}}\]
be given by $O_n(\{a_1,\ldots,a_n\})=(a_{\tau(1)},a_{\tau(2)},\ldots, a_{\tau(n)})$ where $\tau$ is the unique permutation of $\{1,2,\ldots, n\}$ such that $a_{\tau(1)}<a_{\tau(2)}< \ldots < a_{\tau(n)}$. As in dimension 2 we use $\proj_i:\underbrace{\N \times \cdots \times \N}_{n\textrm{-times}} \to \N$ to be the projection onto the $i$th coordinate and also use $\proj_i:\Delta^n \to [0,1]$ to be projection onto the $i$th coordinate, $x_i$. When $P$ is an $n$-pattern of size $N$ and $F \subseteq \Delta^n$, a \emph{copy of $P$ in $F$} is a function 
\[\sigma: O_n(P) \to F\]
so that for any $a,b \in O_n(P)$ and any $p,q \in \{\proj_1,\proj_2,\ldots,\proj_n\}$
\[p(a)<q(b) \implies p(\sigma(a))<q(\sigma(b)).\]
Notice that a copy $\sigma$ of $P$ induces an order preserving function $f_\sigma:[N] \to [0,1]$
\[f_\sigma(m)=\proj_i(\sigma(a))  \]
where $i \leq n$ and $a$ is such that $m=\proj_i(a)$.

Given $0\leq a<b\leq 1$, we let 
\[D_{a,b}^n=\{x \in \Delta^n \ : \ \proj_i(x) \in (a,b) \textrm{ for all }i\leq n\}\]
and call each such set $D_{a,b}^n$ an \emph{essential $n$-simplex}. An $n$-pattern $P$ is \emph{essential for $F$} if there is a copy of $P$ in $F \cap D^n_{a,b}$ for each $a<b$.

\begin{conj}\label{conj1}
Let $F \subseteq \Delta^n$. Then at least one of the following holds:
\begin{enumerate}
    \item $\rank(F^c)=\infty$
    \item any finite $n$-pattern is essential for $F$
\end{enumerate}
\end{conj}

An argument very similar to that presented in Section 2 gives that Conjecture 1 implies that for any $n\in \N$, every minimal $\Homeo_+[0,1]$-subflow of $\Exp(\Delta^n)$ is a singleton that consists of a union of faces of $\Delta^n$. Therefore a proof of Conjecture \ref{conj1} (which is independent of Pestov's theorem) would produce a new proof of the extreme amenability of $\Homeo_+[0,1]$.

Here is the theorem in dimension 3 that represents partial progress towards Conjecture 1. Theorem \ref{thm3} considers the first ``non-trivial'' pattern in 3-dimensions.

\begin{thm}\label{thm3}
Let $F \subseteq \Delta^3$. Then at least one of the following holds:
\begin{enumerate}
\item $\rank(F^c)=\infty$
\item the pattern $\{0,2,4\},\{1,3,5\}$ is essential for $F$
\end{enumerate}
\end{thm}
From now on, we will only be working with $\Delta^3$ and with $3$-patterns; when we say \emph{pattern} we always mean a $3$-pattern.

First, a bit of notation. For $j_1,j_2,l_1,l_2 \in \N$  with $j_1 < j_2$ define:
\[\Phi_{(j_1,l_1),(j_2,l_2)}:\N \to \N\]
by
\[
\Phi_{(j_1,l_1),(j_2,l_2)}(x) =
\begin{cases}
x & \textrm{ if } x < j_1\\
x+l_1 &\textrm{ if }j_1 \leq x < j_2\\
x+l_1+l_2 & \textrm{ if }x\geq j_2\\
\end{cases}
\]

We also define
\[\Phi_{(j_1,l_1)}:\N \to \N\]
by
\[
\Phi_{(j_1,l_1)}(x) =
\begin{cases}
x & \textrm{ if } x < j_1\\
x+l_1 &\textrm{ if }x \geq j_1\\
\end{cases}
\]
We use also the convention that
\[\Phi_{(j_1,l_1),(j_1,l_2)}=\Phi_{(j_1,l_1+l_2)}.\]
Every map defined above is injective and order-preserving on its domain. Abusing notation, we will use $\Phi_{(j_1,l_i),(j_2,l_2)}$ and $\Phi_{(j_1,l_1)}$ to denote the obvious corresponding map on $2^{\N}$, the collection of subsets of $\N$, and the obvious coordinate-wise map on $\N\times \N \times \N$ as well.

For $k \in \N$ we let
\[T_k: \N \to \N\]
be the translation given by
\[T_k(i)=i+k\]
and we use $T_k$ also to denote the obvious extensions to $2^{\N}$ and $\N \times \N \times \N$.

Now suppose that $P$ is a pattern of size $n$, $j_1\leq j_2 \leq n$, $Q$ is a pattern of size $m$, and $i\leq m$. Then when $j_1<j_2$, by $P_{j_1}^{j_2} \oplus Q_i$ we mean the pattern:
\[\Phi_{(j_1,i),(j_2,m-i)}(P) \cup \Phi_{(0,j_1),(i,j_2-j_1)}(Q)\]

A consequence of pattern $Q$ of size $n$ being essential for a set $F$ is that for any pattern $P$ which is essential for $F$ and any $l_1<l_2\leq m$ where $m$ is the size of $P$, the pattern
\[P_{l_1}^{l_2} \oplus Q_{n}\]
is essential for $F$. We will use this easy observation implicitly many times in the proofs below.

The key construction is that of a \emph{chain}. Fix a pattern $P$ of size $n$ and $j<k\leq n$. We inductively define a \emph{$l$-chain of $P$ at $j,k$}, which we denote by $\textrm{Ch}^l_{j,k}(P)$ as follows. A 1-chain of $P$ at $j,k$ is just $P$
\[\textrm{Ch}^1_{j,k}=P\]
Having defined $\textrm{Ch}^l_{j,k}(P)$, define 
\[\textrm{Ch}^{l+1}_{j,k}(P)=\left(\textrm{Ch}^l_{j,k}(P)\right)_{lj}^{(l-1)j+k} \oplus P_k\]
It is helpful to be able to keep track of the individual copies of $P$ that form a chain of $P$. Suppose that $\sigma: O(\textrm{Ch}_{j,k}^l(P)) \to \Delta^2$ is a copy of the $l$-chain of $P$ at $j,k$ and that $P$ is of size $n$. Then for $i=1,\ldots, l$ we denote by $\sigma^{(i)}$ the copy of $P$ added in stage $i$ of the inductive construction of the chain. To be precise, define $\sigma^{(i)}: O(P) \to \Delta^2$ by:
\[\sigma^{(1)}=\sigma \circ \Phi_{(j,(l-2)(n)+k),(k,n-k)}\]
and for $i=2,\ldots, l$:
\begin{equation}\label{eqn4}
   \sigma^{(i)}= \sigma \circ (T_j)^{i-1} \circ \Phi_{(j,(l-i-1)(n)+k), (k,n-j)} 
\end{equation}

One downside to the notation is that it tends to obscure the geometric intuition of patterns ``glued'' together in specific formations that underlies nearly all of the arguments to come. The reader is encouraged to draw diagrams while reading the proofs; this will make the arguments much easier to follow. To demonstrate, suppose that $P$ and $Q$ are arbitrary patterns of size $m$ and $n$ respectively, and $j_1<j_2\leq n$ and $k\leq m$, Figure \ref{fig1} is a diagram of $P_{j_1}^{j_2} \oplus Q_k$. 

\begin{figure}[ht]
\centering
\includegraphics[height=2.8cm]{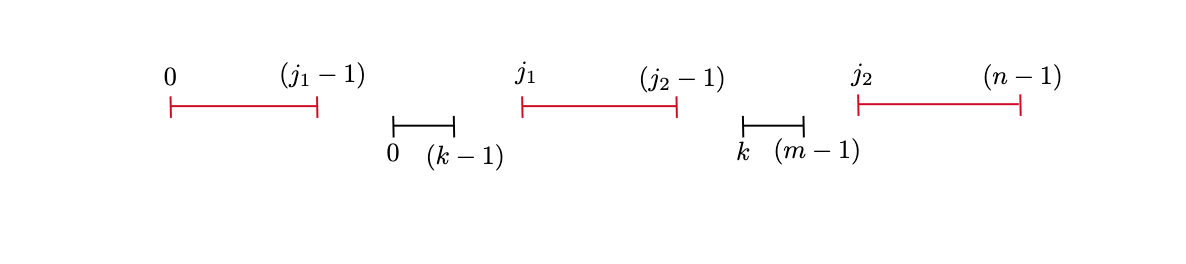}
\caption{A diagram of $P_{j_1}^{j_2} \oplus Q_k$; the pattern $P$ is in red and $Q$ is in black.}
\label{fig1}
\end{figure}

\begin{lem}\label{lem5}
Let $P$ be a pattern of size $n$ with $j< k< n$. Assume $F\subseteq \Delta^2$ is so that $\rank(F^c)<\infty$. If $\textrm{Ch}^m_{j,k}(P)$ is essential for $F$ for every $m \in \N$, then
\[P' = \Phi_{(j,1),(k,1)}(P) \cup \{(j,k+1,n+2)\}\]
is essential for $F$.
\end{lem}

\begin{proof}
Let $m$ be such that every set $S$ with $\rank(S) \geq \lfloor\frac{m}{2}\rfloor$ intersects $F$. Let $a<b$ and let $\sigma:O\left(\textrm{Ch}^m_{j,k}(P)\right) \to F \cap D_{a,b}$ be a copy of $\textrm{Ch}^m_{j,k}(P)$. Let $f_\sigma$ be the order-preserving function $[mn] \to [0,1]$ induced by $\sigma$.

Let $f_{\sigma^{(i)}}:[n] \to [0,1]$ be the order-preserving functions induced by $\sigma^{(i)}$, the $i$th copy of $P$ in the chain as defined above.

Now consider the set:
\[S=\bigcup_{i=1}^m (f_{\sigma^{(i)}}(j-1),f_{\sigma^{(i)}}(j)) \times (f_{\sigma^{(i)}}(k-1),f_{\sigma^{(i)}}(k)) \times (f_{\sigma^{(i)}}(n-1),b)\]

\begin{claim}\label{claim2}
$\rank(S) \geq \lfloor \frac{m}{2} \rfloor$.
\end{claim}

Once we have Claim \ref{claim2}, we are done; our assumption on $F$ implies that $F \cap S \neq \emptyset$ and in particular, there exists 
\[(x,y,z) \in F \cap (f_{\sigma^{(i)}}(j-1),f_{\sigma^{(i)}}(j)) \times (f_{\sigma^{(i)}}(k-1),f_{\sigma^{(i)}}(k)) \times (f_{\sigma^{(i)}}(n-1),b)\]
Then, we can extend copy $\sigma^{(i)}$ of $P$ to a copy $(\sigma^{(i)})'$ of $P'$ in $F \cap D_{a,b}$ by:
\[(\sigma^{(i)})'(u)=
\begin{cases}
\sigma^{(i)}(v) & \textrm{ if }u=\Phi_{(j,1),(k,1)}(v)\\
(x,y,z) & \textrm{ otherwise}\\
\end{cases}\]
So it remains to show the Claim.
\end{proof}

\begin{proof}[Proof of Claim \ref{claim2}]
Let $x_0=f_{\sigma^{(m)}}(j-1)$, $x_1=f_{\sigma^{(m)}}(j)$, $y_1=f_{\sigma^{(m)}}(k-1)$, and $z_1=f_{\sigma^{(m)}}(n-1)$. For $2\leq i\leq \lfloor\frac{m}{2}\rfloor$, let $x_i=f_{\sigma^{(m-2(i-1))}}(j)$, $y_i=f_{\sigma^{(m-2(i-1))}}(k-1)$, and $z_i=f_{\sigma^{(m-2(i-1))}}(n-1)$. 

Since each $f_{\sigma^{(i)}}$ is order-preserving, it is clear that $x_i \leq y_i \leq z_i$ for each $i$. We check that for each $i$, $z_i<x_{i+1}$. First we have:
\begin{align*}
    z_i &= f_{\sigma^{(m-2i+2)}}(n-1)\\
    &=f_\sigma \left((T_j)^{m-2i+1} \circ \Phi_{(j,(m-(m-2i+2)-1)(n)+k),(k,n-j)}(n-1)\right)\\
    &= f_{\sigma}\left((T_j)^{m-2i+1} (n-1+(2i-3)(n)+k+n-j)\right)\\
    &=f_{\sigma}\left((2i-1)(n)+k-j-1+j(m-2i+1)\right)
\end{align*}
The second equality above is using Equation \ref{eqn4}. A similar computation shows that
\[x_{i+1}=f_\sigma ((2i-1)(n)+k+j+j(m-2i-1))\]
Since 
\[(2i-1)(n)+k+j+j(m-2i-1)-((2i-1)(n)+k-j-1+j(m-2i+1))=1>0\]
and $f_\sigma$ is order-preserving, we have that $x_{i+1}<z_i$.
Clearly also $a<x_0<x_1$ and $z_{\lfloor \frac{m}{2}\rfloor}<b$. 

We now must check that
\begin{equation}\label{eqn1}
\textrm{int}\left( \bigcup_{0\leq p<q<r\leq m} [x_p,x_{p+1}] \times [y_q,y_{q+1}] \times [z_r,z_{r+1}] \right) \subseteq S
\end{equation}

For $i=1,2,\ldots,m$, let 
\[C_i= (f_{\sigma^{(i)}}(j-1),f_{\sigma^{(i)}}(j)) \times (f_{\sigma^{(i)}}(k-1),f_{\sigma^{(i)}}(k)) \times (f_{\sigma^{(i)}}(n-1),b). \]
\noindent To check \ref{eqn1}, it suffices to check that 
\[[x_p,x_{p+1}] \times  [y_q,y_{q+1}] \times [z_r,z_{r+1}] \subseteq C_{m-2q-2} \cup C_{m-2q-1}\]
when $0<p$, $p+1<q$, $q+1<r$, and $r <m$. We leave this computation to the reader to check using Equation \ref{eqn4} repeatedly.

Analogous containments hold:   \[[x_p,x_{p+1}) \times (y_q,y_{q+1}] \times [z_r,z_{r+1}] \subseteq C_{m-2q-2} \cup C_{m-2q-1}\]
when $0<p$, $p+1=q$, $q+1<r$, and $r<m$ and  
\[(x_p,x_{p+1}] \times [y_q,y_{q+1}) \times (z_r,z_{r+1}] \subseteq C_{m-2q-2} \cup C_{m-2q-1}\]
when $p=0$, $p+1<q$, $q+1=r$, $r<m$ and so on. It follows that $\rank(S) \geq \lfloor \frac{m}{2}\rfloor$.
\end{proof}

The lemma below captures a specific situation in which one has a high rank set:

\begin{lem}\label{lem6}
Let $a_1,b_1,c_1,a_2,b_2,c_2, \ldots,a_m,b_m,c_m$ be such that for each $i$, $a_1<a_{i+1}<b_{i+1}<c_{i+1}<b_i$ and $b_1<d$ and define
\[S=\bigcup_{i=1}^m (a_i,b_i) \times (a_i,b_i) \times (c_i,d)\]
Then, $\rank(S) \geq m$.
\end{lem}

\begin{proof}
We will choose inductively numbers
\[x_0,x_1,y_1,z_1,\ldots, x_m,y_m,z_m,z_{m+1}\]
that witness that rank of $S$ is at least $m$ as in Definition \ref{def2}. Let $x_0=a_m$, and choose $x_1=y_1 \in (a_m,b_m)$. Then, choose $z_1 \in (c_m,b_{m-1})$. Having chosen $x_{i-1}$, $y_{i-1}$, $z_{i-1}$, choose $x_i=y_i \in (z_{i-1},b_{m-i+1})$ and then choose $z_i \in (c_{m-i+1},b_{m-i})$. Set $z_{m+1}=d$. Notice that
\[x_0<x_1 =y_1 <z_1 <x_2=y_2 <z_2<\cdots <x_m=y_m<z_m<z_{m+1}\]
Checking that these values witness that $\rank(S) \geq m$ is similar to checking the containments at the end of the proof of Lemma \ref{lem5}, so we omit the proof.
\end{proof}

Lemma \ref{lem6} has the following corollary that allows one to augment patterns in a certain way.

\begin{cor}\label{cor1}
Assume $F$ is so that $\rank(F^c)<\infty$. Let $P$ be an essential pattern for $F$ of size $n$ and let $k< n$. Then the pattern
\[\Phi_{(k,2)}(P) \cup \{(k,k+1,n+2)\}\]
is essential for $F$.
\end{cor}

\begin{proof}
Let $M$ be so that every set of rank at least $M$ intersects $F$. Let $P$ be an essential for $F$ pattern of size $n$. Fix $0\leq e<f\leq 1$. Set $P'=\Phi_{(k,2)}(P) \cup \{(k,k+1,n+2)\}$; we want to show that there is a copy of $P'$ in $F \cap D_{e,f}$.

We will inductively define patterns $K_1, \ldots, K_M$. Let $K_1=P$. Given $K_i$, define
\[K_{i+1}=\left(K_i\right)_{(ik,n)}\oplus P_{n}\]
Clearly $K_1$ is essential and
\[K_i \textrm{ essential }\implies K_{i+1} \textrm{ essential }\]
because $P$ is essential. So $K_M$ is essential.

Now let $\sigma: O(K_M) \to F \cap D_{e,f}$ be a copy of $K_M$ and let $f_\sigma :[Mn] \to [0,1]$ be the associated order-preserving function. We let $\sigma^{(i)}: O(P) \to F \cap D_{e,f}$ be the copy of $P$ in $\sigma$ added in the construction of $K_i$ from $K_{i-1}$.

For $i=1, \ldots, M$ set:
\begin{align*}
    & a_i=f_\sigma (ik-1)=f_{\sigma^{(i)}}(k-1)\\
    & b_i=f_{\sigma} (ik+n(M-i)-1)=f_{\sigma^{(i)}}(k)\\
    & c_i=f_\sigma ((M-i+1)(n)+(i-1)(k)-1)=f_{\sigma^{(i)}}(n-1)\\
\end{align*}
and set $d=f$.

Now Lemma \ref{lem6} implies that the set
\[S= \bigcup_{i=1}^M (a_i,b_i) \times (a_i,b_i) \times (c_i,d)\]
has rank at least $M$ and so intersects $F$. This means there exists $j\leq m$ and $(x,y,z) \in F$ so that $x,y \in (a_j,b_j)$ and $z \in (c_j,d)$. 

Then, one checks that 
\[\sigma':O(P') \to F\]
given by 

\[
\sigma'(u)=
\begin{cases}
 f_{\sigma_j}(v) & \textrm{ if }u=\Phi_{(k,2)}(v)\\
 (x,y,z) & \textrm{ if }u=(k,k+1,n+2)\\
\end{cases}
\]
is a copy of $P'$ in $F \cap D_{e,f}$.
\end{proof}

We can now prove Theorem \ref{thm3}.

\begin{proof}[Proof of Theorem \ref{thm3}]
Assume that $F \subset \Delta^2$ is such that $\rank(F^c) <\infty$. Throughout the proof of the Theorem, whenever we say that a pattern is ``essential," we mean that it is ``essential for $F$". Let $M \in \N$ so that each set of rank at least $M$ intersects $F$. We want to show that $\{0,2,4\},\{1,3,5\}$ is essential. Let $P$ be the singleton pattern $\{0,1,2\}$. We prove the following statement by induction on $m$:

\begin{claim}\label{claim3}
For each $m \in \N$, if $Q$ is any essential pattern (of size $n$) and $k \leq n$, then the pattern
\[\left(\textrm{Ch}^m_{1,2}(P)\right)_{m-1}^{m+1}\oplus Q_k\]
is essential.
\end{claim}
\begin{proof}[Proof of Claim \ref{claim3}]
The base case of $m=1$ is Corollary \ref{cor1}.

Suppose that the Claim holds for some $m \in \N$. We want to prove the Claim for $m+1$. So fix an essential pattern $Q$ of size $n$ and $k\leq n$. Fix also $0\leq e<f\leq 1$; we want to show that there is a copy of
\[\left(\ch_{1,2}^{m+1}(P)\right)_{m}^{m+2}\oplus Q_k\]
in $F \cap D_{e,f}$.

First we will construct by induction patterns $K_i$ for $i=1, 2,\ldots, M$ as follows.

Set 
\[K_1=\left(\ch_{1,2}^m(P)\right)_{m}^{3m} \oplus Q_{n}\]
and set

\[K_2=\left( \ch_{1,2}^m(P)_{m-1}^{m+1} \oplus \left(K_1\right)_{m+k}\right)_{2m+k}\oplus Q_{n}\]
and continue on in this way, so that, precisely:
\[K_i=\left(\ch_{1,2}^m(P)_{m-1}^{m+1} \oplus \left(K_{i-1}\right)_{(i-1)(m+k)}\right)_{(i)(m)+(i-1)(k)}\oplus Q_{n}\]

We have that $K_1$ is essential because $\ch_{1,2}^m(P)$ and $Q$ are essential. Given that $K_{i-1}$ is essential,
\[\ch_{1,2}^m(P)_{m-1}^{m+1} \oplus \left(K_{i-1}\right)_{(i-1)(m+k)}\]
is essential by an application of the induction hypothesis (that the Claim holds at $m$) and the fact that $K_{i-1}$ is essential. Then:
\[K_i=\left(\ch_{1,2}^m(P)_{m-1}^{m+1} \oplus \left(K_{i-1}\right)_{(i-1)(m+k)}\right)_{(i)(m)+(i-1)(k)}\oplus Q_{n}\]
is essential by the fact that $Q$ is essential.

So, in fact $K_M$ is essential. Take a copy $\sigma:O(K_M) \to F \cap D_{e,f}$. For $i=1,\ldots,M$, let $\sigma^{(i)}: O(P') \to F \cap D_{e,f}$
where 
\[P'=\left(\ch_{1,2}^m(P)\right)_{m}^{3m}\oplus Q_{n}\]
and $\sigma^{(i)}$ is the copy of $P'$ added to $K_{i-1}$ to form $K_i$ within $\sigma$. To be precise,
\[\sigma^{(1)} = \sigma \circ \Phi_{(0,(m-1)(M-1)), (m+k,(n+2)(M-1))}\]
and for $i=2,\ldots, M$
\[\sigma^{(i)}= \sigma \circ T_{(M-i)(m-1)} \circ \Phi_{(m-1, (i-1)(m+k)), (m+n+1, (i-1)(2m+n-k))}\]

For $i=1,\ldots,M$, set
\begin{align*}
    & a_i=f_{\sigma^{(i)}}(m+k-1)\\
    & b_i=f_{\sigma^{(i)}}(m+k)\\
    & c_i=f_{\sigma^{(i)}}(m+n)\\
\end{align*}
and set $d=f_{\sigma^{(1)}}(m+n+1)$. One can check that the conditions of Lemma \ref{lem6} are satisfied and so there exists $1\leq i\leq M$ and $(x,y,z) \in F$ so that $x,y \in (a_i,b_i)$ and $z \in (c_i,d)$. That is:
\[x,y \in (f_{\sigma^{(i)}}(m+k-1), f_{\sigma^{(i)}}(m+k))\]
and 
\[z \in (f_{\sigma^{(i)}}(m+n),d) \subseteq (f_{\sigma^{(i)}}(m+n), f_{\sigma^{(i)}}(m+n+1))\]
Now there is a copy $\tau$ of the pattern $\left(\ch_{1,2}^{m+1}(P)\right)_{m}^{m+2}\oplus Q_k$ in $F \cap D_{e,f}$ given by

\[
\tau(u)=
\begin{cases}
 \sigma_i(v) & \textrm{ if }u=\Phi_{(m+k,2),(m+n+1,1)}(v)\\
 (x,y,z) & \textrm{ if }u= (m+k,m+k+1,m+n+3)\\
\end{cases}
\]
\end{proof}
Claim \ref{claim3} with $Q$ taken to be the empty pattern obviously implies that for any $m$, $\ch_{1,2}^m(P)$ is essential; so by Lemma \ref{lem5}, the pattern $\{0,2,4\},\{1,3,5\}$ is essential.
\end{proof}

\end{document}